\documentclass[11pt]{amsart}
\usepackage{a4wide,amssymb}

\newcommand{\GEN}{\mathsf{G_{EN}}}
\newcommand{\SGEN}{\mathsf{\dot G_{EN}}}
\newcommand{\GNE}{\mathsf{G_{NE}}}
\newcommand{\BM}{\mathsf{BM}}
\newcommand{\w}{\omega}
\newcommand{\Ra}{\Rightarrow}

\newtheorem{theorem}{Theorem}
\newtheorem{problem}{Problem}
\newtheorem{proposition}{Proposition}

\theoremstyle{definition}

\newtheorem{remark}{Remark}

\title{Some Baire category properties of topological groups}
\author{Taras Banakh and Olena Hryniv}
\address{T.Banakh: Institute of Mathematics, Jan Kocahnowski University in Kielce (Poland) and Ivan Franko National University of Lviv (Ukraine)}
\email{t.o.banakh@gmail.com}
\address{O. Hryniv: Ivan Franko National University of Lviv, Ukraine}
\email{ohryniv@gmail.com}
\dedicatory{Dedicated to the 60th birthday of our teacher M.M. Zarichnyi}

\begin{document}
\begin{abstract} We present several known and new results on the Baire category properties in topological groups. In particular, we prove that a Baire topological group $X$ is metrizable if and only if $X$ is point-cosmic if and only if $X$ is a $\sigma$-space. A topological group $X$ is Choquet if and only if its Raikov completion $\bar X$ is Choquet and $X$ is $G_\delta$-dense in $\bar X$. A topological group $X$ is complete-metrizable if and only if  $X$ is a point-cosmic Choquet space if and only if $X$ is a Choquet $\sigma$-space.  Finally, we pose several open problem, in particular, whether each Choquet topological group is strong Choquet.
\end{abstract}

\maketitle

In this paper we collect some results on Baire category properties of topological groups and pose some open problems.

The best known among Baire category properties are the meagerness and Baireness defined as follows.

A topological space $X$ is defined to be 
\begin{itemize}
\item {\em meager} if it can be written as the countable union of closed subsets with empty interior in $X$;
\item {\em Baire} if for any sequence $(U_n)_{n\in\w}$ of open dense subsets fo $X$ the intersection $\bigcap_{n\in\w} U_n$ is dense in $X$.
\end{itemize}
It is well-known that a topological space is Baire if and only if it contains no non-empty open meager subspace. This fact and the topological homogeneity of topological groups imply the following known characterization, see \cite[9.8]{Ke}.

\begin{proposition}\label{p:M} A topological group is Baire if and only if it is not meager.
\end{proposition}

In framework of topological groups, the Baireness combined with a suitable network property often implies the metrizability.

A family $\mathcal N$ of subsets of a topological space $X$ is called
\begin{itemize}
\item a {\em network} at a point $x\in X$ if for any neighborhood $O_x\subset X$ of $x$ the union $\bigcup\{N\in\mathcal N:N\subset O_x\}$ is a neighborhood of $x$;
\item a {\em network} if $\mathcal N$ is a network at each point.
\end{itemize}

A regular topological space $X$ is called
\begin{itemize}
\item {\em point-cosmic} if $X$ has a countable network $\mathcal N_x$ at each point $x\in X$;
\item {\em cosmic} if $X$ has a countable network;
\item a {\em $\sigma$-space} of $X$ has a $\sigma$-discrete network.
\end{itemize}
It is easy to see that each first-countable regular space is point-cosmic, and each cosmic space is both point-cosmic and a $\sigma$-space. The sequential fan $S_{\w_1}$ with uncountably many spikes is an example of a $\sigma$-space, which is not point-cosmic. More information on cosmic and $\sigma$-spaces can be found in \cite[\S 4]{Grue}.

\begin{theorem}\label{t:B} For a Baire topological group $X$ the following conditions are equivalent:
\begin{enumerate}
\item $X$ is metrizable;
\item $X$ is point-cosmic;
\item $X$ is a $\sigma$-space.
\end{enumerate}
\end{theorem}

\begin{proof}  The implications $(1)\Ra(2,3)$ are trivial (or well-known). 
\smallskip

$(3)\Ra(1)$ Assume that $X$ is a $\sigma$-space. By Theorems 5.4 and 7.1 in \cite{BanP}, the Baire $\sigma$-space $X$ contains a dense metrizable subspace $M\subset X$. The regularity of $X$ implies that $X$ is first-countable at each point $x\in M$. Being topologically homogeneous, the topological group $X$ is first-countable and by   Birkhoff-Kakutani Theorem \cite[9.1]{Ke}, $X$ is metrizable. 
\smallskip

$(2)\Ra(1)$  Assume that the topological group $X$ is point-cosmic and hence $X$ has a countable network $\mathcal N$ at the unit $e$. For every $N\in\mathcal N$ let $(N^{-1}\bar N)^\circ$ be the interior of the set $N^{-1}\bar N=\{x^{-1}y:x\in N,\;y\in\bar N\}$ in $X$. Here $\bar N$ stands for the closure of the set $N$ in $X$. 

We claim that the family
$$\mathcal B:=\{(N^{-1}\bar N)^\circ:N\in\mathcal N,\;\;e\in(N^{-1}\bar N)^\circ\}$$is a countable neighborhood base at the unit $e$ of the group $X$. Given any neighborhood $U\subset X$ of $e$, find an open neighborhood $V\subset X$ of $e$ such that $V^{-1}VV^{-1}\subset U$. By the definition of a network at the point $e$, the union $\bigcup\mathcal N_V$ of the subfamily $\mathcal N_V:=\{N\in\mathcal N:N\subset V\}$ contains some open neighborhood $O_e$ of $e$ in $X$. Since $O_e$ is Baire (as an open subset of the Baire space $X$), some set $N\in\mathcal N_V$ is not nowhere dense in $V$. Consequently, its closure $\bar N$ contains a non-empty open subset $W$ of $X$. Since $W\subset \bar N$, the intersection $W\cap N$ contains some point $w$. Then $w^{-1}W$ is an open neighborhood of $e$ such that $$w^{-1}W\subset N^{-1}\bar{N}\subset V^{-1}\bar V\subset V^{-1}VV^{-1}\subset U.$$ Consequently, $w^{-1}W\subset (N^{-1}\bar N)^\circ$ and $(N^{-1}\bar N)\in\mathcal B$.

Being first-countable, the topological group $X$ is metrizable by the Birkhoff-Kakutani Theorem \cite[9.1]{Ke}. 
\end{proof}

Both Baire and meager spaces have nice game characterizations due to Oxtoby \cite{Oxtoby}. To present these characterizations, we need to recall the definitions of two infinite games $\GEN(X)$ and $\GNE(X)$ played by two players $\mathsf E$ and $\mathsf N$ (abbreviated from $\mathsf{Empty}$ and $\mathsf{Non}$-$\mathsf{Empty}$) on a topological space $X$.

The game $\GEN(X)$ is started by the player $\mathsf E$ who selects a non-empty open set $U_0\subset X$. Then player $\mathsf N$ responds selecting a non-empty open set $U_1\subset U_0$. At the $n$-th inning the player $\mathsf E$ selects a non-empty open set $U_{2n}\subset U_{2n-1}$ and the player $\mathsf N$ responds selecting a non-empty open set $U_{2n+1}\subset U_{2n}$. At the end of the game, the player $\mathsf E$ is declared the winner if $\bigcap_{n\in\w}U_n$ is empty. In the opposite case the player $\mathsf N$ wins the game $\GEN(X)$.

The game $\GNE(X)$ differs from the game $\GEN(X)$ by the order of the players. The  game $\GNE(X)$ is started by the player $\mathsf N$ who selects a non-empty open set $U_0\subset X$. Then player $\mathsf E$ responds selecting a non-empty open set $U_1\subset U_0$. At the $n$-th inning the player $\mathsf N$ selects a non-empty open set $U_{2n}\subset U_{2n-1}$ and the player $\mathsf E$ responds selecting a non-empty open set $U_{2n+1}\subset U_{2n}$. At the end of the game, the player $\mathsf E$ is declared the winner if $\bigcap_{n\in\w}U_n$ is empty. In the opposite case the player $\mathsf N$ wins the game $\GNE(X)$.

The following classical characterization can be found in  \cite{Oxtoby}.

\begin{theorem}[Oxtoby]\label{t:Oxtoby} A topological space $X$ is 
\begin{itemize}
\item meager if and only if the player $\mathsf E$ has a winning strategy in the game $\GNE(X)$;
\item Baire if and only if the player $\mathsf E$ has no winning strategy in the game $\GEN(X)$.
\end{itemize}
\end{theorem}
A topological space $X$ is defined to be {\em Choquet} if the player $\mathsf N$ has a winning strategy in the Choquet game $\GEN(X)$. Choquet spaces were introduced in 1975 by White \cite{White} who called them {\em weakly $\alpha$-favorable spaces}.

Oxtoby's Theorem~\ref{t:Oxtoby} implies that $$\mbox{Choquet $\Ra$ Baire $\Ra$ non-meager}.$$

Let us also recall the notion of a strong Choquet space, which is defined using a modification $\SGEN(X)$ of the Choquet game $\GEN(X)$, called the strong Choquet game. 

The game $\SGEN(X)$ is played by two players, $\mathsf E$ and $\mathsf N$ on a topological space $X$. The player $\mathsf E$ starts the game selecting a open set $U_0\subset X$ and a point $x_0\in U_0$. Then the player $\mathsf N$ responds selecting an open neighborhood $U_1\subset U_0$ of $x_0$. At the $n$-th inning the player $\mathsf E$ selects an open set $U_{2n}\subset U_{2n-1}$ and a point $x_n\in U_{2n}$ and player $\mathsf N$ responds selecting a neighborhood $U_{2n+1}\subset U_{2n}$ of $x_n$. At the end of the game the player $\mathsf E$ is declared the winner if the intersection $\bigcap_{n\in\w}U_n$ is empty. Otherwise the player $\mathsf N$ wins the game $\SGEN(X)$. 

A topological space $X$ is called {\em strong Choquet} if the player $\mathsf N$ has a winning strategy in the game $\SGEN(X)$. More information on (strong) Choquet spaces can be found in \cite[8.CD]{Ke}. Various topological games are analyzed in \cite{Gru} and \cite{Tel}. 

It is known (and easy to see) that the class of strong Choquet spaces includes all spaces that are homeomorphic to complete metric spaces. Such spaces will be called {\em complete-metrizable}.

So, for any topological space $X$ we have the implications
$$\mbox{complete-metrizable $\Ra$ strong Choquet $\Ra$ Choquet $\Ra$ Baire $\Ra$ non-meager}.$$

By \cite{Choq} and \cite{Oxtoby} (see also \cite[8.17]{Ke}), a metrizable space is
\begin{itemize}
\item strong Choquet if and only if it is complete-metrizable;
\item Choquet if and only if it contains a dense complete-metrizable subspace.
\end{itemize}

In the following theorem we present a characterization of Choquet topological groups. A subset $A$ of a topological space $X$ is called {\em $G_\delta$-dense} in $X$ if $A$ has non-empty intersection with any non-empty $G_\delta$-set $G\subset X$.

\begin{theorem}\label{t:C} A topological group $X$ is Choquet if and only if its Raikov completion $\bar X$ is Choquet and $X$ is $G_\delta$-dense in $\bar X$.
\end{theorem}

\begin{proof} The ``if'' part easily follows from the definitions of the Choquet game and the $G_\delta$-density of $X$ in $\bar X$. 

To prove the ``only if'' part, assume that a topological group $X$ is Choquet. Then its Raikov completion $\bar X$ also is Choquet (since $\bar X$ contains a dense Choquet subspace). It remains to prove that $X$ is $G_\delta$-dense in $\bar X$. Assuming that this is not true, we could find a point $\bar x\in \bar X$ and a sequence $(O_n)_{n\in\w}$ of neighborhoods of $\bar x$ such that $X\cap \bigcap_{n\in\w}O_n=\emptyset$. Choose a decreasing sequence $(\Gamma_n)_{n\in\w}$ of open neighborhoods of the unit in $\bar X$ such that $\Gamma_n^{-1}\Gamma_n\subset \Gamma_{n-1}$, and $\bar x\Gamma_n\subset O_n$ for every $n\in\w$. Then the intersection $\Gamma:=\bigcap_{n\in\w}\Gamma_n$ is a subgroup in $\bar X$ such that the set
$$g\Gamma=\bigcap_{n\in\w}g\Gamma_n\subset \bigcap_{n\in\w}O_n$$is disjoint with the subgroup $X$. This implies $X\cap X\bar x\Gamma=\emptyset$ and $X\Gamma\cap X\bar x=\emptyset$. 

Given a subset $A\subset \bar X$ consider the Banach-Mazur game $\BM(\bar X,A)$ played by two players, $\mathsf E$ and $\mathsf N$ according to the following rules. The player $\mathsf E$ starts the game selecting a non-empty open set $U_0\subset \bar X$ and the player $\mathsf N$ responds selecting a non-empty open set $U_1\subset U_0$. At the $n$-th inning the player $\mathsf E$ chooses a non-empty open set $U_{2n}\subset U_{2n-1}$ and the player $\mathsf N$ responds selecting a non-empty open set $U_{2n+1}\subset U_{2n}$. At the end of the game the player $\mathsf E$ is declared the winner if the intersection $A\cap\bigcap_{n\in\w}U_n$ is empty. In the opposite case the player $\mathsf N$ wins the game.

It is easy to see that for any dense Choquet subspace $A\subset \bar X$ the player $\mathsf N$ has a winning strategy in the Banach-Mazur game $\BM(\bar X,A)$. Moreover, at $n$th inning the player $\mathsf N$ can replace the set $U_{2n}$ selected by the player $\mathsf E$ by a set $U_{2n}'\subset U_{2n}$ so small that $U'_{2n}\subset x_n\Gamma_n$ for some $x_n\in X$. Then the set $U_{2n+1}$ chosen by player $\mathsf N$ according to his/her strategy will be contained in $x_n\Gamma_n$.

In particular, the player $\mathsf N$ has a winning strategy in the game $\BM(\bar X,X)$ having this smalleness property. By the same reason, the player $\mathsf N$ has a winning strategy in the Banach-Mazur game $\BM(\bar X,X\bar x)$. Now the players $\mathsf E$ and $\mathsf N$ can use these winning strategies in the games $\BM(X,H)$ and $\BM(X,Hg)$ to construct a decreasing sequence $(U_n)_{n\in\w}$ of non-empty open sets in $\bar X$ such that the intersection $\bigcap_{n\in\w}U_n$ meets both sets $X$ and $X\bar x$. Moreover, in the $n$th inning the player $\mathsf N$ can make his/sets $U_{2n}$ so small that $U_{2n}\subset x_n\Gamma_n$ for some $x_n\in X$. Choose two points $x\in H\cap \bigcap_{n\in\w}U_n$ and $y\in Hg\cap\bigcap_{n\in\w}U_n$. Then for every $n\in\w$ we have $x\in U_{2n}\subset x_n\Gamma_n$ and hence $x_n\in x\Gamma_n^{-1}$ and $U_{2n}\subset x_n\Gamma_n\subset x\Gamma_n^{-1}\Gamma_n\subset x\Gamma_{n-1}$.
Consequently, $y\in \bigcap_{n\in\mathbb N}U_{2n}\subset\bigcap_{n\in\mathbb N}x\Gamma_{n-1}=x\Gamma$ and hence $y\in X\bar x\cap x\Gamma\subset X\bar x\cap X\Gamma=\emptyset$, which is a desired contradiction.
\end{proof}

It is easy to see that each $G_\delta$-dense subspace of a strong Choquet space is strong Choquet.
 
\begin{problem} Is the Raikov-completion of a strong Choquet topological group strong Choquet?
\end{problem}

Using Theorems~\ref{t:B} and \ref{t:C} we can prove the following characterization of complete-metrizable topological groups.

\begin{theorem}\label{t:CM} For a topological group $X$ the following conditions are equivalent:
\begin{enumerate}
\item $X$ is complete-metrizable;
\item $X$ is a point-cosmic Choquet space;
\item $X$ is Choquet $\sigma$-space.
\end{enumerate}
\end{theorem}

\begin{proof}  The implications $(1)\Ra(2,3)$ are trivial.

$(2,3)\Ra(1)$ Assuming that a Choquet topological group $X$ is point-cosmic or a $\sigma$-space, we can apply Theorem~\ref{t:B} and conclude that the Choquet (and hence Baire) topological group $X$ is metrizable. Then its Raikov completion $\bar X$ is a complete-metrizable topological group, see \cite[9.A]{Ke}. By Theorem~\ref{t:C}, $X$ is $G_\delta$-dense in $\bar X$. Since each seingleton in $\bar X$ is a $G_\delta$-set, the $G_\delta$-density of $X$ in $\bar X$ implies that $X=\bar X$ and hence $X$ is a complete-metrizable (and complete) topological group. 
\end{proof}

\begin{problem}\label{prob1} Is each Choquet topological group strong Choquet?
\end{problem}

\begin{remark} By Theorem~\ref{t:CM}, the answer to Problem~\ref{prob1} is affirmative for topological groups which are point-cosmic or $\sigma$-spaces.
\end{remark}

It is known \cite[8.13]{Ke} (and \cite[8.16]{Ke}) that the product of two (strong) Choquet space is (strong) Choquet. On the other hand, the product $X\times Y$ of two Baire spaces can be meager, see \cite{vMP}. 

\begin{problem} Let $H$ be a closed normal subgroup of a topological group $G$ and $G/H$ be the quotient topological group. 
Assume that $H$ and $G/H$ are (strong) Choquet. Is $G$ (strong) Choquet?
\end{problem}

\noindent{\bf Acknowledgement.} The authors would like to thank 
Gabriel Andre Asmat Medina for clarifying the situation with the product of countably cellular Baire spaces.

\end{document}